\documentclass[12pt]{article}

\usepackage{amsmath,amssymb,amsthm,amscd,amsfonts}
\usepackage{epsfig,pinlabel,paralist}
\usepackage{mathtools}
\usepackage[vmargin=1in, hmargin=1.25in]{geometry}
\usepackage[font=small,format=plain,labelfont=bf,up,textfont=it,up]{caption}
\usepackage{calc}
\usepackage{etoolbox}
\usepackage{microtype}
\usepackage[all,cmtip]{xy}

\usepackage[bookmarks, bookmarksdepth=2, colorlinks=true, linkcolor=blue, citecolor=blue, urlcolor=blue]{hyperref}

\apptocmd{\thebibliography}{\raggedright}{}{}

\numberwithin{equation}{section}

\theoremstyle{plain}
\newtheorem{theorem}{Theorem}[section]
\newtheorem{maintheorem}{Theorem}

\newtheorem{lemma}[theorem]{Lemma}

\newtheorem{corollary}[theorem]{Corollary}

\newtheorem{claims}{Claim}

\theoremstyle{definition}

\newtheorem{defn}[theorem]{Definition}

\theoremstyle{remark}
\newtheorem{rmk}[theorem]{Remark}
\newenvironment{remark}[1][]{\begin{rmk}[#1] \pushQED{}}{\popQED \end{rmk}}

\newtheorem{eg}[theorem]{Example}

\DeclareMathOperator{\Hom}{Hom}
\DeclareMathOperator{\Diff}{Diff}
\DeclareMathOperator{\Emb}{Emb}

\DeclareMathOperator{\GL}{GL}
\newcommand\GLp{\GL^{+}}

\DeclareMathOperator{\SO}{SO}

\newcommand\R{\ensuremath{\mathbb{R}}}

\newcommand\Z{\ensuremath{\mathbb{Z}}}

\DeclareMathOperator{\HH}{H}
\newcommand\RH{\ensuremath{\widetilde{\HH}}}

\DeclareMathOperator{\RP}{\mathbb{RP}}

\DeclareMathOperator{\Aut}{Aut}
\DeclareMathOperator{\Out}{Out}

\newcommand\Set[2]{\ensuremath{\left\{\text{#1 $|$ #2}\right\}}}

\newcommand\cD{\ensuremath{\mathcal{D}}}

\newcommand\fD{\ensuremath{\mathfrak{D}}}

\newcommand\fT{\ensuremath{\mathfrak{T}}}

\newcommand\tM{\ensuremath{\widetilde{M}}}

\newcommand\tf{\ensuremath{\widetilde{f}}}

\newcommand\tx{\ensuremath{\widetilde{x}}}

\newcommand\tfD{\ensuremath{\widetilde{\fD}}}

\newcommand\oQ{\ensuremath{\overline{Q}}}

\newcommand\Figure[4]{
\begin{figure}[t]
\centering
\centerline{\psfig{file=#2,scale=#4}}
\caption{#3}
\label{#1}
\end{figure}}

\DeclareMathOperator{\ab}{ab}
\DeclareMathOperator{\Mod}{Mod}
\DeclareMathOperator{\Fr}{Fr}
\DeclareMathOperator{\Triv}{Triv}
\DeclareMathOperator{\HTriv}{HTriv}
\DeclareMathOperator{\Tw}{Twist}
\newcommand\TwNosep{\ensuremath{\Tw_{\text{ns}}}}
\DeclareMathOperator{\Sphere}{\mathbb{S}}
\newcommand\tGLp{\widetilde{\GL}^{+}}

\title{\vspace{-40pt}The mapping class group of connect sums of $S^2 \times S^1$\vspace{-15pt}}
\author{Tara Brendle \and Nathan Broaddus \and Andrew Putman\thanks{Supported in part by NSF grant DMS-1811210}}
\date{}

\begin{document}

\vspace{-10pt}
\maketitle

\vspace{-18pt}
\begin{abstract}
\noindent
Let $M_n$ be the connect sum of $n$ copies of $S^2 \times S^1$.  A classical theorem of Laudenbach says
that the mapping class group $\Mod(M_n)$ is an extension of $\Out(F_n)$ by a group $(\Z/2)^n$ generated
by sphere twists.  We prove that this extension splits, so $\Mod(M_n)$ is the semidirect
product of $\Out(F_n)$ by $(\Z/2)^n$, which $\Out(F_n)$ acts on via the dual of the natural surjection
$\Out(F_n) \rightarrow \GL_n(\Z/2)$.  Our splitting takes $\Out(F_n)$ to the subgroup of $\Mod(M_n)$
consisting of mapping classes that fix the homotopy class of a trivialization of the tangent bundle of $M_n$.
Our techniques also simplify various aspects of Laudenbach's original proof, including the identification 
of the twist subgroup with $(\Z/2)^n$. 
\end{abstract}

\section{Introduction}
\label{section:introduction}

The {\em mapping class group} of a closed oriented $3$-manifold $M^3$, denoted $\Mod(M^3)$, is the
group of isotopy classes of orientation-preserving diffeomorphisms of $M^3$.  In this paper,
we study the mapping class group of the connect sum $M_n$ of $n$ copies of $S^2 \times S^1$.
The fundamental group $\pi_1(M_n)$
is the free group $F_n$ on $n$ letters.  Since
diffeomorphisms of $M_n$ do not fix a basepoint, the action of $\Mod(M_n)$ on $\pi_1(M_n)$ is only
defined up to conjugation, so it gives a homomorphism
\[\rho\colon \Mod(M_n) \longrightarrow \Out(\pi_1(M)) \cong \Out(F_n).\]
It follows from work of Whitehead \cite{Whitehead1, Whitehead2} that $\rho$ is surjective.  Its kernel
was described by Laudenbach \cite{LaudenbachSpheres, LaudenbachBook}.

\paragraph{Sphere twists.}
The kernel of $\rho$ is the subgroup generated by sphere twists, which are defined as follows.  Let
$M^3$ be a closed oriented $3$-manifold and let
$S \subset M^3$ be a smoothly embedded $2$-sphere.  Fix a tubular neighborhood $U \cong S \times [0,1]$
of $S$.  Recall that $\pi_1(\SO(3),\text{id}) \cong \Z/2$ is generated by a loop $\ell\colon [0,1] \rightarrow \SO(3)$
that rotates $\R^3$ about an axis by one full turn.  Identifying $S$ with $S^2 \subset \R^3$, the {\em sphere twist}
about $S$, denoted $T_S$, is the isotopy class of the diffeomorphism $\tau\colon M^3 \rightarrow M^3$ that
is the identity outside $U$ and on $U \cong S \times [0,1]$ takes the form $\tau(s,t) = (\ell(t) \cdot s,t)$.
The isotopy class of $T_S$ only depends on the isotopy class of $S$.  In fact,
Laudenbach \cite{LaudenbachSpheres, LaudenbachBook}
proved that if $S$ and $S'$ are homotopic $2$-spheres in $M^3$ that are non-nullhomotopic, then
$S$ and $S'$ are isotopic, so $T_S$ actually only depends on the homotopy class of $S$.

\paragraph{Actions of sphere twists.}
We have $\SO(3) \cong \RP^3$, so $\pi_1(\SO(3),\text{id}) \cong \Z/2$.  It follows that the mapping class $T_S$ has order at most $2$.  It
follows from Laudenbach's work that in the case $M^3 = M_n$, the sphere twist $T_S$ is trivial if 
$S$ separates $M_n$ and has order
$2$ if $S$ is nonseparating.  Showing that $T_S$ is ever nontrivial is quite subtle since $T_S$ fixes
the homotopy class of any loop or surface $Z$ in $M^3$, and thus cannot be detected by
most basic algebro-topological invariants.  To see this, let $U \cong S \times [0,1]$ be the tubular
neighborhood used to construct $T_S$ and let $p_0 \in S$ be one of the two points of $S$ lying on the axis
of rotation used to construct $T_S$.  Homotope $Z$ to be transverse to $S$. 
The intersection $Z \cap S$ is then either a collection
of circles (if $\dim(Z)=2$) or points (if $\dim(Z)=1$).  As is shown in Figure \ref{figure:gather}, we can
then homotope $Z$ such that $Z \cap U \subset p_0 \times [0,1]$, so $T_S$ fixes $Z$.

\Figure{figure:gather}{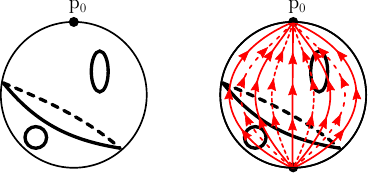}{On the left is $S$ along with $p_0$ and $Z \cap S$ (which here
is $1$-dimensional, so $\dim(Z)=2$).  On the right we show how to homotope $Z$ such
that $Z \cap S = \{p_0\}$ -- choose a point $q \in S \setminus (\{p_0\} \cup Z)$, and homotope
$Z$ so as to push its intersection with $S$ along paths from $q$ to $p_0$ until it is entirely
contained in $p_0$.}{95}

\paragraph{Twist subgroup.}
The {\em twist subgroup} of $\Mod(M^3)$, denoted $\Tw(M^3)$, is the subgroup generated by all sphere twists.
For $f \in \Mod(M^3)$ and a sphere twist $T_S \in \Tw(M^3)$, we have
\[f T_S f^{-1} = T_{f(S)}.\]
This implies that $\Tw(M^3)$ is a normal subgroup of $\Mod(M^3)$.  Also, if $T_S$ and $T_{S'}$ are sphere twists,
we saw above that $T_{S'}(S)$ is homotopic to $S$.  Since a sphere twist only depends on the homotopy class
of the sphere along which we are twisting, setting $f = T_{S'}$ in the above relation we get
\[T_{S'} T_S T_{S'}^{-1} = T_S.\]
In other words, $\Tw(M^3)$ is abelian.

\paragraph{Laudenbach sequence.}
Laudenbach's theorem can thus be summarized as saying that there is a short exact sequence
\begin{equation}
\label{eqn:laudenbachseq}
1 \longrightarrow \Tw(M_n) \longrightarrow \Mod(M_n) \stackrel{\rho}{\longrightarrow} \Out(F_n) \longrightarrow 1.
\end{equation}
He also proved that $\Tw(M_n) \cong (\Z/2)^n$ and is generated by the sphere twists about the core
spheres $S^2 \times \ast$ of the $n$ different $S^2 \times S^1$ summands of $M_n$.  
This theorem raises two natural questions:
\begin{compactenum}
\item Does the extension \eqref{eqn:laudenbachseq} split?
\item The conjugation action of $\Mod(M_n)$ on its normal abelian group $\Tw(M_n) \cong (\Z/2)^n$ induces
an action of $\Out(F_n)$ on $\Tw(M_n) \cong (\Z/2)^n$.  What action is this?
\end{compactenum}

\paragraph{Main theorem.}
Our main theorem answers both of these questions.  It says that the extension \eqref{eqn:laudenbachseq}
does split, and in fact the image of the splitting $\Out(F_n) \rightarrow \Mod(M_n)$ has a simple geometric
description: it is the stabilizer of the homotopy class of a trivialization of the tangent bundle of $M_n$.
A precise statement is as follows.

\begin{maintheorem}
\label{maintheorem:split}
Let $[\sigma_0]$ be the homotopy class of a trivialization $\sigma_0$ of the tangent bundle of $M_n$
and let $(\Mod(M_n))_{[\sigma_0]}$ be the $\Mod(M_n)$-stabilizer of $[\sigma_0]$.  The following then hold:
\begin{compactitem}
\item $\Mod(M_n) = \Tw(M_n) \rtimes (\Mod(M_n))_{[\sigma_0]}$.
\item $\Tw(M_n) \cong \HH^1(M_n;\Z/2)$ as a $\Mod(M_n)$-module.
\item $(\Mod(M_n))_{[\sigma_0]} \cong \Out(F_n)$.
\end{compactitem}
\end{maintheorem}

\begin{remark}
Before Laudenbach's work, Gluck \cite{GluckOriginal} proved that $\Mod(M_1) \cong \Z/2 \times \Z/2$.  The first factor
was a sphere twist, and the second factor was $\Out(F_1) = \Z/2$.  This is of course a special case of
Theorem \ref{maintheorem:split}.
\end{remark}

\paragraph{Sphere complex and $\Out(F_n)$.}
Laudenbach's exact sequence \eqref{eqn:laudenbachseq} plays an important role in the study
of $\Out(F_n)$.  In his seminal paper \cite{HatcherOriginal}, Hatcher defined the {\em sphere complex}
$\Sphere_n$ to be the simplicial complex whose $k$-simplices are sets $\{S_0,\ldots,S_k\}$ of 
isotopy classes of non-nullhomotopic smoothly embedded $2$-spheres in $M_n$ that can be realized
disjointly.  One of his main theorems says that $\Sphere_n$ is contractible.  The group $\Mod(M_n)$
acts on $\Sphere_n$, and since sphere twists fix the isotopy class of any smoothly embedded $2$-sphere
the twist subgroup $\Tw(M_n)$ acts trivially.  By \eqref{eqn:laudenbachseq}, we thus get
an action of $\Out(F_n)$ on $\Sphere_n$.  

The space $\Sphere_n$ is also sometimes called the {\em free splitting complex}
and has played an important role in a huge amount of subsequent work (see, e.g.,
\cite{AramayonaSouto, ClayQingRafi, HamenstadtHensel, HandelMosher, HatcherVogtmann, HilionHorbez, Horbez, KapovichLustig}).
It is unsatisfying that the original construction of the action of $\Out(F_n)$ on
$\Sphere_n$ was so indirect: you first construct an action of $\Mod(M_n)$, then
notice that $\Tw(M_n)$ acts trivially, and only then get an induced action
of the quotient group $\Out(F_n) = \Mod(M_n)/\Tw(M_n)$.  It follows from
Theorem \ref{maintheorem:split} that $\Out(F_n)$ can be embedded as a subgroup
of $\Mod(M_n)$, so there is no longer a need to perform this indirect
construction.

\paragraph{Nontriviality of sphere twists.}
Our proof also gives a new and easier argument for seeing that the sphere twists $T_S$ about nonseparating
spheres in $M_n$ (and other $3$-manifolds) are nontrivial (cf.\ Corollary~\ref{corollary:identifytwist}).  This is not as easy as one might expect.  The usual
way that one studies a group like $\Mod(M_n)$ is via its action on homotopy classes of submanifolds of $M_n$.  However, the
sphere twists $T_S$ fix the homotopy classes of all loops and surfaces in $M_n$, so a new idea is needed.  Laudenbach used
an argument involving framed cobordism and the Pontryagin--Thom construction, while as we describe below we study
the action of $\Mod(M_n)$ on trivializations of the tangent bundle.  

\begin{remark}
The idea of using the Pontryagin--Thom construction to study sphere-twists in $3$-manifolds goes
back to early work of Pontryagin; see the example at the end of \cite[\S 4]{PontryaginThree}.
\end{remark}

\paragraph{Derivative crossed homomorphism.}
The heart of our proof of Theorem \ref{maintheorem:split} is what we call the derivative crossed
homomorphism.  This is a cocycle on $\Mod(M_n)$ constructed using the action of $\Mod(M_n)$ on
the set of trivializations of the tangent bundle of $M_n$.  A similar idea (using stable trivializations
rather than trivializations) has been used to study mapping class groups of high-dimensional manifolds,
where it is used to understand various short exact sequences coming from surgery theory.  The
earliest appearance of this idea seems to be work of Krylov \cite{Krylov}, and it was further
developed by Crowley \cite{Crowley} and Krannich \cite{Krannich}.

\paragraph{Automorphisms vs outer automorphisms.}
Let $M_{n,1}$ be $M_n$ equipped with a basepoint $\ast \in M_{n}$, and define $\Mod(M_{n,1})$ to be
the group of isotopy classes of orientation-preserving diffeomorphisms of $M_{n,1}$ that fix $\ast$.
The group $\Mod(M_{n,1})$ then acts on $\pi_1(M_{n,1},\ast) = F_n$, so we get a homomorphism
$\Mod(M_{n,1}) \rightarrow \Aut(F_n)$.  It follows from Laudenbach's work that we also have a short
exact sequence
\[1 \longrightarrow \Tw(M_{n,1}) \longrightarrow \Mod(M_{n,1}) \longrightarrow \Aut(F_n) \longrightarrow 1\]
and that $\Tw(M_{n,1}) = (\Z/2)^n$.  Our work shows that this sequence also splits, and a result
identical to Theorem \ref{maintheorem:split} holds.  This can be proved either by adapting our proof (which
needs almost no changes), or by using the exact sequence
\begin{equation}
\label{eqn:birman}
1 \longrightarrow F_n \longrightarrow \Mod(M_{n,1}) \longrightarrow \Mod(M_n) \longrightarrow 1
\end{equation}
arising from the long exact sequence in homotopy groups of the fiber bundle
\[\Diff^{+}(M_n,\ast) \longrightarrow \Diff^{+}(M_n) \longrightarrow M_n.\]
The $F_n$ in \eqref{eqn:birman} is the image of $\pi_1(M_n)$ in $\Mod(M_{n,1}) = \pi_0(\Diff^{+}(M_n,\ast))$ and
maps to the inner automorphisms in $\Aut(F_n)$.  We leave the details to the interested reader.

\paragraph{Other $3$-manifolds.}
Let $M^3$ be an arbitrary closed orientable $3$-manifold and let $\pi = \pi_1(M^3)$.  The twist subgroup $\Tw(M^3)$ is still
an abelian normal subgroup of the mapping class group $\Mod(M^3)$, and it turns out that
there is still an exact sequence
\begin{equation}
\label{eqn:generalsequence}
1 \longrightarrow \Tw(M^3) \longrightarrow \Mod(M^3) \longrightarrow G \longrightarrow 1,
\end{equation}
where $G < \Out(\pi)$ is the image of $\Mod(M^3)$ in $\Out(\pi)$.  See \cite[Proposition 2.1]{HatcherWahl}
for how to extract this from the literature; in fact, this reference also gives an appropriate
but more complicated statement for $3$-manifolds with boundary.  It is often the case that
$G = \Out(\pi)$; see \cite[Proposition 2.2]{HatcherWahl} for some conditions that ensure this.

In light of Theorem \ref{maintheorem:split}, it is natural to wonder whether \eqref{eqn:generalsequence}
splits.  Let $\TwNosep(M^3)$ be the subgroup of $\Tw(M^3)$ generated by sphere twists about
nonseparating spheres. Our proof of Theorem \ref{maintheorem:split} can be generalized to show
that we have a splitting
\[\Mod(M^3) = \TwNosep(M^3) \rtimes \Gamma,\]
where $\Gamma$ is a subgroup of $\Mod(M^3)$.   However, we cannot show that $\Gamma$ can be taken
to be the stabilizer of the homotopy class of a trivialization of $T M^3$ (c.f.\ the
proof of Theorem \ref{theorem:weakmain} below).  

Unfortunately, we do not know how to deal
with separating sphere twists.  For $M^3 = M_n$, separating sphere twists are always trivial,
so $\TwNosep(M^3) = \Tw(M^3)$ and separating twists can be ignored when studying
\eqref{eqn:generalsequence}.  However, for general $3$-manifolds the situation is very complicated.  Let $S \subset M^3$
be a $2$-sphere that separates $M^3$.
\begin{itemize}
\item In \cite{Hendriks}, Hendriks gives a remarkable characterization of when $T_S$
is homotopic to the identity.  Namely, $T_S$ is homotopic to the identity if and only if for one
of the two components $N$ of cutting $M^3$ open along $S$ the following strange condition holds.
Let $P$ be a prime summand of the result of gluing a closed $3$-ball to $\partial N = S^2$.
Then either $P = S^2 \times S^1$, or $P$ has a finite fundamental group whose Sylow $2$-subgroup
is cyclic.
\item However, this is not the whole story.  In \cite{FriedmanWitt}, Friedman--Witt show that in
some cases the separating sphere twists that Hendriks showed were homotopic to the identity
are not {\em isotopic} to the identity, and thus still define nontrivial elements of $\Mod(M^3)$.
\end{itemize}
What happens in the general case is unclear.  See \cite[Remark 2.4]{HatcherWahl} for some
further discussion of it.

\paragraph{Outline.}
The outline of our paper is as follows.  We start in \S \ref{section:extension} by constructing the
exact sequence \eqref{eqn:laudenbachseq}.  To make our paper self-contained, we give a mostly complete proof
of this, simplifying some details from Laudenbach's original paper.  We then have a preliminary
algebraic section \S \ref{section:crossed} on crossed homomorphisms and their relationship
to splitting exact sequences.  We then construct the crossed homomorphisms we need in
\S \ref{section:derivative} and \S \ref{section:twisting} before closing with
\S \ref{section:completing}, which takes care of a few final details.

\paragraph{Notational conventions.}
It will be important for us to distinguish between a diffeomorphism and its isotopy class.  For $f \in \Diff^{+}(M_n)$ we will write $[f] \in \Mod(M_n)$ for the isotopy class of $f$.  More generally, we will use square brackets frequently to indicate that something is being taken up to homotopy/isotopy, though we will try to be explicit about this whenever it might be confusing to the reader.

\paragraph{Acknowledgments.}
We would like to thank Jake Landgraf for helpful conversations and Allen Hatcher for some useful
references, historical remarks, and corrections.  We would also like to thank Oscar Randal-Williams
for pointing out the work of Crowley \cite{Crowley}, Krannich \cite{Krannich}, and Krylov \cite{Krylov}.
We additionally thank Diarmuid Crowley, Dan Margalit, and Richard Wade for comments and corrections.

\section{Constructing the extension}
\label{section:extension}

This preliminary section discusses some aspects of Laudenbach's work we will need for our proof.

\paragraph{What is needed.}
Recall from the introduction that Laudenbach \cite{LaudenbachSpheres} proved that
$\Mod(M_n)$ is an extension of $\Out(F_n)$ by $\Tw(M_n)$ and that $\Tw(M_n) \cong (\Z/2)^n$.
Theorem \ref{maintheorem:split} strengthens this and its proof will give as a byproduct that 
$\Tw(M_n) \cong (\Z/2)^n$, but it will depend on the following piece of Laudenbach's work.

\begin{theorem}[Laudenbach, \cite{LaudenbachSpheres, LaudenbachBook}]
\label{theorem:laudenbach}
The map 
\[\rho\colon \Mod(M_n) \rightarrow \Out(\pi_1(M_n)) = \Out(F_n)\]
is surjective with kernel $\ker(\rho) = \Tw(M_n)$.  Also, $\Tw(M_n)$ is generated
by the sphere twists about the core spheres $S^2 \times \ast$ of the $n$ summands of $S^2 \times S^1$ in $M_n$.
\end{theorem}

\begin{remark}
\label{remark:laudenbachnontrivial}
Theorem \ref{theorem:laudenbach} does {\em not} assert that the indicated sphere twists
are nontrivial.  As we discussed in the introduction, Laudenbach proved that they are and
that $\Tw(M_n) \cong (\Z/2)^n$, but we will only establish this (in a stronger form) later; see
Corollary \ref{corollary:identifytwist}.
\end{remark}

To make this paper more self-contained, this section contains a mostly complete sketch of a proof of
Theorem~\ref{theorem:laudenbach}.  We will follow the outline of Laudenbach's original proof,
but we will simplify one key step (see Theorem \ref{theorem:pi2} below).

\paragraph{Homotopy vs isotopy for spheres.}
Our proof of Theorem \ref{theorem:laudenbach} will depend on three preliminary results.  The first
is the following:

\begin{theorem}[Laudenbach, \cite{LaudenbachSpheres, LaudenbachBook}]
\label{theorem:sphereisotopy}
Let $M^3$ be a closed oriented $3$-manifold and let $\iota,\iota'\colon \sqcup_{i=1}^k S^2 \rightarrow M^3$ be
homotopic embeddings of disjoint smoothly embedded spheres.  Assume that none of the components of
the images of $\iota$ or $\iota'$ are nullhomotopic.  Then $\iota$ and $\iota'$ are ambient isotopic.
\end{theorem}

We omit the proof of Theorem \ref{theorem:sphereisotopy} since it is lengthy and its details do not
shed much light on our work.

\paragraph{Action on second homotopy group.}
The second preliminary result is the following theorem of
Laudenbach.  Our proof is much shorter than his proof.
We remark that Hatcher-Wahl have given a different (but somewhat
longer) simplified proof in \cite[Appendix]{HatcherWahl}.

\begin{theorem}[Laudenbach]
\label{theorem:pi2}
Let $M^3$ be a closed oriented $3$-manifold equipped with a basepoint $x_0$ 
and let $f\colon (M^3,x_0) \rightarrow (M^3,x_0)$ be a basepoint-preserving diffeomorphism
such that $f_{\ast}\colon \pi_1(M^3,x_0) \rightarrow \pi_1(M^3,x_0)$ is the identity.  Then
$f_{\ast}\colon \pi_2(M^3,x_0) \rightarrow \pi_2(M^3,x_0)$ is the identity.
\end{theorem}
\begin{proof}
Let $(\tM^3,\tx_0) \rightarrow (M^3,x_0)$ be the universal cover of $(M^3,x_0)$.  Let
$\tf\colon (\tM^3,\tx_0) \rightarrow (\tM^3,\tx_0)$ be the lift of $f$.  To prove that
$f$ acts trivially on $\pi_2(M^3,x_0)$, it is enough to prove that $\tf$ acts trivially
on $\pi_2(\tM^3,\tx_0) = \HH_2(\tM^3)$.  By Poincar\'{e} duality, we have
\begin{equation}
\label{eqn:identifyh2}
\HH_2(\tM^3) \cong \HH^1_c(\tM^3) = \lim_{\substack{\rightarrow\\ K}} \HH^1(\tM^3,\tM^3 \setminus K) = \lim_{\substack{\rightarrow\\ K}} \RH^0(\tM^3 \setminus K).
\end{equation}
Here the limit is over compact subspaces $K$ of $\tM^3$ and the final equality comes from the long exact sequence of the
pair $(\tM^3,\tM^3 \setminus K)$ and the fact that $\tM^3$ is $1$-connected.  Elements of $\RH^0(\tM^3 \setminus K)$ can
be interpreted as locally constant functions 
$\kappa\colon \tM^3 \setminus K \rightarrow \Z$ modulo the globally constant functions.  Fix such a
$\kappa\colon \tM^3 \setminus K \rightarrow \Z$, and let $K'$ be a compact subspace of $\tM^3$ containing
$K \cup \tf(K)$ such that no components of $\tM^3 \setminus K'$ are bounded (i.e., have compact closure).
The image under the homeomorphism $\tf$ of the element of 
$\HH_2(\tM^3)$ represented by $\kappa$ under \eqref{eqn:identifyh2} is represented by the function
\[\kappa \circ \tf^{-1}\colon \tM^3 \setminus K' \rightarrow \Z.\]
We must prove that $\kappa = \kappa \circ \tf^{-1}$ on $\tM^3 \setminus K'$.  The key observation
is that since $f$ acts trivially on $\pi_1(M^3,x_0)$, the lift $\tf$ fixes each point in the
$\pi_1(M^3,x_0)$-orbit of the basepoint $\tx_0$.  This orbit will contain points in each component
of $\tM^3 \setminus K'$, so $\kappa$ and $\kappa \circ \tf^{-1}$ agree on at least
one point in each component of $\tM^3 \setminus K'$.  Since they are locally constant,
we conclude that they are equal everywhere on $\tM^3 \setminus K'$, as desired.
\end{proof}

\paragraph{Mapping class groups of punctured spheres.}
The third and final preliminary result we need is as follows.  For a $3$-manifold $M^3$ with boundary, we define
$\Mod(M^3)$ to be $\pi_0(\Diff^{+}(M^3,\partial M^3))$, i.e., the group of isotopy classes of
orientation-preserving diffeomorphisms of $M^3$ that fix $\partial M^3$ pointwise.

\begin{lemma}
\label{lemma:spheretwists}
Let $X$ be the $3$-manifold with boundary obtained by removing $k$ disjoint open balls from $S^3$.  Then
$\Mod(X)$ is generated by sphere twists about embedded spheres that are parallel to components of $\partial X$.
\end{lemma}

\begin{remark}
Just like Theorem \ref{theorem:laudenbach}, Lemma \ref{lemma:spheretwists} does {\em not} assert
that these sphere twists are nontrivial in the mapping class group.  In fact, one can show that they
are trivial if $k=1$ and nontrivial if $k \geq 2$, and the twist subgroup
of the manifold $X$ in Lemma \ref{lemma:spheretwists} is isomorphic to $(\Z/2)^{k-1}$.  Here $(k-1)$
appears instead of $k$ since the product of all the boundary twists is trivial;
see \cite[p.\ 214--215]{HatcherWahl}.
We will not need any of this, so we will not prove it.
\end{remark}

\begin{proof}[Proof of Lemma \ref{lemma:spheretwists}]
The proof will be by induction on $k$.  The base case $k=0$ simply asserts that $\Mod(S^3) = 1$, which is
a theorem of Cerf \cite{CerfConnected}.  We remark that even more is true: the $3$-dimensional Smale Conjecture
proved by Hatcher \cite{HatcherSmale} says that $\Diff^{+}(S^3)$ is homotopy equivalent to $\SO(4)$.  Assume
now that $k>0$ and that the lemma is true for smaller $k$.  Let $X'$ be the result of gluing a closed $3$-ball
$B$ to a component of $\partial X$.  We thus have $\Diff^{+}(X,\partial X) = \Diff^{+}(X',\partial X' \sqcup B)$.
Let $\Emb^{+}(B,X')$ be the space of orientation-preserving embeddings of $B$ into the interior of $X'$.  Restricting elements of
$\Diff^{+}(X',\partial X')$ to $B$ gives a map 
\[\Diff^{+}(X',\partial X') \rightarrow \Emb^{+}(B,X')\]
that fits into a fiber bundle
\[\Diff^{+}(X',\partial X' \sqcup B) \rightarrow \Diff^{+}(X',\partial X') \rightarrow \Emb^{+}(B,X').\]
Identifying $\Mod(X)$ and $\Mod(X')$ with $\pi_0$ of the relevant diffeomorphism groups, 
the associated long exact sequence in homotopy groups contains the segment
\begin{equation}
\label{eqn:inductiveemb}
\pi_1(\Emb^{+}(B,X')) \longrightarrow \Mod(X) \longrightarrow \Mod(X') \longrightarrow \pi_0(\Emb^{+}(B,X')).
\end{equation}
Fix oriented trivializations of the tangent bundles of $B$ and $X'$.  For an orientation-preserving
embedding $\iota\colon B \rightarrow X'$, these
trivializations allow us to identify the derivative $D_0 \iota\colon T_0 B \rightarrow T_{\iota(0)} X'$ with
a matrix in the subgroup $\GLp_3(\R)$ of $\GL_3(\R)$ consisting of matrices whose determinant is positive.  The map $\Emb^{+}(B,X') \rightarrow X' \times \GLp_3(\R)$ taking
an embedding $\iota\colon \rightarrow X'$ to $(\iota(0), D_0 \iota)$ is a homotopy equivalence (see, e.g., \cite[Theorem 9.12]{KupersBook}).
The group $\GLp_3(\R)$ deformation retracts onto its maximal compact subgroup 
$\SO(3) \cong \RP^3$, so we deduce that
\[\pi_0(\Emb^{+}(B,X')) = \pi_0(X' \times \GLp_3(\R)) = 0\]
and
\[\pi_1(\Emb^{+}(B,X')) = \pi_1(X' \times \GLp_3(\R)) = \pi_1(\GLp_3(\R)) = \Z/2.\]
Plugging these into \eqref{eqn:inductiveemb}, we get an exact sequence
\[\Z/2 \longrightarrow \Mod(X) \longrightarrow \Mod(X') \longrightarrow 0.\]
The image of $\Z/2$ in $\Mod(X)$ is a sphere twist about a sphere parallel to a component of $\partial X$, and
by induction $\Mod(X')$ is generated by sphere twists about spheres parallel to components of $\partial X'$.  The
lemma follows.
\end{proof}

\paragraph{The proof.}
We now have all the ingredients needed for the proof of Theorem \ref{theorem:laudenbach} above.

\begin{proof}[Proof of Theorem \ref{theorem:laudenbach}]
Recall that
\[\rho\colon \Mod(M_n) \rightarrow \Out(\pi_1(M_n)) = \Out(F_n)\]
is the natural map.  We must prove the following two facts.

\begin{claims}
\label{claims:surjective}
The map $\rho$ is surjective.
\end{claims}

\Figure{figure:glueup}{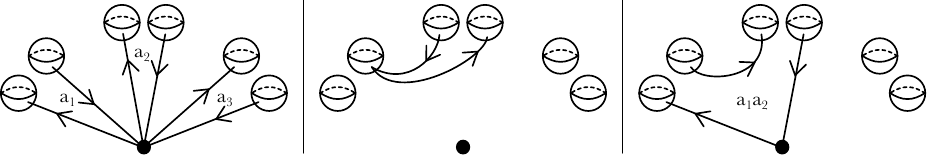}{On the left hand side, $M_3$ is obtained by gluing the $6$ boundary components
of $X$ together in pairs as indicated.  The generators for $\pi_1(M_3) = F_3$ are $\{a_1,a_2,a_3\}$.  In the
middle, we indicate a diffeomorphism $\phi\colon M_3 \rightarrow M_3$ that drags one boundary sphere of $X$ along
a closed path.  As is shown on the right, on $\pi_1(M_3)$ the diffeomorphism $\phi$ takes $a_1$ to $a_1 a_2$.}{95}

Let $X$ be the result of removing $2n$ disjoint open balls from $S^3$.  As in
Figure \ref{figure:glueup}, we can identify $M_n$ with the result of gluing the boundary
component of $X$ together in pairs.  Let $\{a_1,\ldots,a_n\} \in \pi_1(M_n) = F_n$ be
the generators indicated in Figure \ref{figure:glueup}.  An old theorem of Nielsen (\cite{NielsenGen}; see \cite[\S I.4]{LyndonSchupp} for a textbook
reference) says that the group $\Aut(F_n)$ is generated
by the following elements:
\begin{compactitem}
\item For distinct $1 \leq i,j \leq n$, elements $L_{ij}$ and $R_{ij}$ defined via the formulas
\[L_{ij}(a_k) = \begin{cases}
a_j a_k & \text{if $k=i$}\\
a_k & \text{if $k \neq i$}
\end{cases}
\quad \text{and} \quad
R_{ij}(a_k) = \begin{cases}
a_k a_j & \text{if $k=i$}\\
a_k & \text{if $k \neq i$}
\end{cases}
\quad \quad (1 \leq k \leq n).\]
\item For $1 \leq i \leq n$, elements $I_i$ defined via the formula
\[I_i(a_k) = \begin{cases}
a_k^{-1} & \text{if $k=i$},\\
a_k & \text{if $k \neq i$}.
\end{cases}\]
\end{compactitem}
It is enough to find elements of $\Diff^{+}(M_n)$ realizing these automorphisms.  This is an easy exercise; for
instance, as we show in Figure \ref{figure:glueup} we can realize $R_{12}$ as a diffeomorphism that drags
one boundary sphere of $X$ through another.  We remark that this surjectivity was originally proved by
Whitehead \cite{Whitehead1, Whitehead2}, who used the more complicated generating set of ``Whitehead
transformations''.

\begin{claims}
The kernel of $\rho$ is the twist subgroup $\Tw(M_n)$, and $\Tw(M_n)$ is generated
by the sphere twists about the core spheres of the $n$ summands of $S^2 \times S^1$ in $M_n$.
\end{claims}

Clearly $\Tw(M_n) \subset \ker(\rho)$, so it is enough to prove that every element of $\ker(\rho)$ is a
product of sphere twists about the core spheres $S_1,\ldots,S_n$ of the $S^2 \times S^1$ summands of $M_n$.  Consider
some $[f] \in \ker(\rho)$, and let $\iota\colon \sqcup_{i=1}^n S^2 \rightarrow M_n$ be the embedding
of those core spheres.  Fix a basepoint $x_0 \in M_n$.
Isotoping $f$, we can assume that $f(x_0) = x_0$ and that $f_{\ast}\colon \pi_1(M_n,x_0) \rightarrow \pi_1(M_n,x_0)$
is the identity.  Theorem \ref{theorem:pi2} then implies that $f$ also induces the identity on $\pi_2(M_n,x_0)$.
It follows that $\iota$ is homotopic to $f \circ \iota$, so by Theorem \ref{theorem:sphereisotopy} we can
isotope $f$ such that $\iota = f \circ \iota$.  Let $X$ be the result of cutting $M_n$ open along
the image of $\iota$, so $X$ is diffeomorphic to the result of removing $2n$ open balls from $S^3$.  Since
$\iota = f \circ \iota$, the mapping class $[f]$ is in the image of the homomorphism $\Mod(X) \rightarrow \Mod(M_n)$
that glues the boundary components back together.  Lemma \ref{lemma:spheretwists} says that $\Mod(X)$ is generated
by sphere twists about spheres parallel to its boundary components.  These map to the $[T_{S_i}]$ in $\Mod(M_n)$, and
the desired result follows.
\end{proof}

\section{Crossed homomorphisms and exact sequences}
\label{section:crossed}

As preparation for proving Theorem \ref{maintheorem:split}, this section reviews the connection between crossed homomorphisms and split exact sequences.  
Let $G$ and $H$ be groups such that $G$ acts on $H$ on the right.  We will write this action
using superscripts: for $g \in G$ and $h \in H$, the image of $h$ under $g$ will be denoted
$h^g$.  
A {\em crossed homomorphism} from $G$ to $H$ is a set map $\lambda\colon G \rightarrow H$
such that
\[\lambda(g_1 g_2) = \lambda(g_1)^{g_2} \lambda(g_2).\]
This implies in particular that 
\[\lambda(1) = \lambda(1^2) = \lambda(1)^1 \lambda(1) = \lambda(1)^2,\]
so $\lambda(1) = 1$.
If the action of $G$ on $H$ is trivial, then this reduces to the definition of a homomorphism.
Just like for an ordinary homomorphism, the kernel $\ker(\lambda) = \Set{$g \in G$}{$\lambda(g)=1$}$
is a subgroup of $G$; however, it is not necessarily a normal subgroup.

As the following standard lemma shows, these are closely related to splittings of short
exact sequences.

\begin{lemma}
\label{lemma:crossedsemidirect}
Let $G$ be a group and let $A \lhd G$ be an abelian normal subgroup, so $G$ acts on $A$ on the right via the formula
\[a^g = g^{-1} a g \quad \quad (a \in A, g \in G).\]
Letting $Q = G/A$, the short exact sequence
\[1 \longrightarrow A \longrightarrow G \longrightarrow Q \longrightarrow 1\]
splits if and only if there exists a crossed homomorphism $\lambda\colon G \rightarrow A$ that restricts to the identity on $A$.
Moreover, if such a $\lambda$ exists, then we can choose a splitting $Q \rightarrow G$ whose image is $\ker(\lambda)$, so
$G = A \rtimes \ker(\lambda)$.
\end{lemma}
\begin{proof}
If the exact sequence splits, then there exists a subgroup $\oQ$ of $G$ projecting isomorphically to $Q$, so we can uniquely
write all $g \in G$ as $g = q a$ with $q \in \oQ$ and $a \in A$.  This allows us to define a set map
$\lambda\colon G \rightarrow A$ via the formula
\[\lambda(q a) = a \quad \quad (q \in \oQ, a \in A).\]
This restricts to the identity on $A$, and is a crossed homomorphism since for $q_1,q_2 \in \oQ$ and $a_1,a_2 \in A$ we have
\begin{align*}
\lambda(q_1 a_1 q_2 a_2) &= \lambda(q_1 q_2 a_1^{q_2} a_2) = \lambda(q_1 q_2 a_1^{q_2 a_2} a_2) = a_1^{q_2 a_2} a_2 = \lambda(q_1 a_1)^{q_2 a_2} \lambda(q_2 a_2).
\end{align*}
We remark that the second equality is where we use the fact that $A$ is abelian.

Conversely, assume that there exists a crossed homomorphism $\lambda\colon G \rightarrow A$ that restricts to the identity on $A$.
Define $\oQ = \ker(\lambda)$, so $\oQ < G$ satisfies $\oQ \cap A = 1$.  To prove the theorem, 
we must prove that the surjection
$\pi\colon G \rightarrow Q$ restricts to an isomorphism $\pi\colon \oQ \rightarrow Q$.  Since $\oQ \cap A = 1$, the projection
$\pi\colon \oQ \rightarrow Q$ is injective, so we must only prove that it is surjective.  Consider $q \in Q$.  We
can find some $g \in G$ such that $\pi(g) = q$.  Since $\lambda(g^{-1}) \in A$, we have
$\pi(\lambda(g^{-1}) g) = \pi(g) = q$, so it is enough to prove that $\lambda(g^{-1}) g \in \oQ = \ker(\lambda)$.  For
this, we compute
\[\lambda(\lambda(g^{-1}) g) = \lambda(\lambda(g^{-1}))^g \lambda(g) = \lambda(g^{-1})^g \lambda(g) = \lambda(g^{-1} g) = 1.\]
Here the second inequality uses the fact that $\lambda$ restricts to the identity on $A$.
\end{proof}

\section{The derivative crossed homomorphism}
\label{section:derivative}

By Lemma \ref{lemma:crossedsemidirect}, to prove that the exact sequence
\[1 \longrightarrow \Tw(M_n) \longrightarrow \Mod(M_n) \longrightarrow \Out(F_n) \longrightarrow 1\]
splits, we must construct a crossed homomorphism $\Mod(M_n) \rightarrow \Tw(M_n)$ that restricts to the
identity on $\Tw(M_n)$.  We will do this in two steps, the first in this section and second in the next.
As we said in the introduction, the idea of our construction goes back to work of
Krylov \cite{Krylov} that was further developed by Crowley \cite{Crowley} and Krannich \cite{Krannich}.

\paragraph{Frame bundle.}
What we do in this section works in complete generality, so let $M^3$ be any closed oriented $3$-manifold.  
Let $TM^3$ be the tangent bundle of $M^3$ and
let $\Fr(TM^3)$ be the principle $\GLp_3(\R)$-bundle of oriented frames of $TM^3$.  Here recall
that $\GLp_3(\R)$ is the subgroup of $\GL_3(\R)$ consisting of matrices whose determinant is positive.
The points of $\Fr(TM^3)$
thus consist of orientation-preserving linear isomorphisms $\tau\colon \R^3 \rightarrow T_p M^3$, where
$p \in M^3$ is a point.  The group $\GLp_3(\R)$ act on $\Fr(TM^3)$ on the right in the usual way: regarding
elements of $\GLp_3(\R)$ as isomorphisms $\R^3 \rightarrow \R^3$, we have
\[\tau \cdot M = \tau \circ M \quad \text{for $\tau\colon \R^3 \rightarrow T_p M^3$ in $\Fr(TM^3)$ and $M \in \GLp_3(\R)$}.\]
This action preserves the fibers of the projection $\Fr(TM^3) \rightarrow M^3$, and its restriction
to each fiber is simply transitive.

\paragraph{Trivializations.}
Since $M^3$ is oriented, its tangent bundle $TM^3$ is trivial.  An {\em oriented trivialization} of $TM^3$
is a continuous section $\sigma\colon M^3 \rightarrow \Fr(TM^3)$ of the bundle $\Fr(TM^3) \rightarrow M^3$.
Let $\Triv(M^3)$ be the set of oriented trivializations of $TM^3$ and let
$C(M^3,\GLp_3(\R))$ be the space of continuous maps $M^3 \rightarrow \GLp_3(\R)$.  The group structure
of $\GLp_3(\R)$ endows $C(M^3,\GLp_3(\R))$ with the structure of a topological group, and $C(M^3,\GLp_3(\R))$
acts continuously on $\Triv(M^3)$ on the right via the formula
\begin{align*}
\sigma \cdot \phi = \Big(p \mapsto \sigma(p) \cdot \phi(p)\Big) \quad &\text{for
$\sigma\colon M^3 \rightarrow \Fr(TM^3)$ in $\Triv(M^3)$}\\ 
&\quad\quad\quad\text{and $\phi\colon M^3 \rightarrow \GLp_3(\R)$ in $C(M^3,\GLp_3(\R))$.}
\end{align*}
The action is also simply transitive.

\paragraph{Diffeomorphism actions.}
For $f \in \Diff^{+}(M^3)$, the derivative $Df$ of $f$ induces a map
$(Df)_{\ast}\colon \Fr(TM^3) \rightarrow \Fr(TM^3)$ defined via the formula
\[(Df)_{\ast}(\tau) = \Big(\left(D_p f\right) \circ \tau\colon \R^3 \rightarrow T_{f(p)} M^3\Big) \quad
\text{for $\tau\colon \R^3 \rightarrow T_p M^3$ in $\Fr(TM^3)$}.\]
Using this, we can define a right action of $\Diff^{+}(M^3)$ on $\Triv(M^3)$ via the following formula, where
we use superscripts to avoid confusing this action with the above action of $C(M^3,\GLp_3(\R))$:
\[\sigma^f = (Df^{-1})_{\ast} \circ \sigma \circ f \quad \text{for $f \in \Diff^{+}(M^3)$ and $\sigma\colon M^3 \rightarrow \Fr(TM^3)$ in $\Triv(M^3)$.}\]
The group $\Diff^{+}(M^3)$ also has a right action on $C(M^3,\GLp_3(\R))$ defined via the formula
\[\phi^f = \phi \circ f \quad \text{for $f \in \Diff^{+}(M^3)$ and $\phi\colon M^3 \rightarrow \GLp_3(\R)$ in $C(M^3,\GLp_3(\R))$.}\]
These three different actions are related by the formula
\begin{align*}
\left(\sigma \cdot \phi\right)^f = \sigma^f \cdot \phi^f \quad &\text{for $f \in \Diff^{+}(M^3)$ and $\phi \in C(M^3,\GLp_3(\R))$ and $\sigma \in \Triv(M^3)$.}
\end{align*}
Both sides of this formula are the element of $\Triv(M^3)$ whose value at a point $p \in M^3$ is the linear
isomorphism $\R^3 \rightarrow T_p M^3$ given by
\[\big(D_{f(p)} f^{-1}\big) \circ \big(\sigma\left(f\left(p\right)\right)\big) \circ \big(\phi\left(f\left(p\right)\right)\big).\]

\paragraph{Derivative crossed homomorphism.}
Our next goal is to construct a crossed homomorphism 
\[\cD\colon \Diff^{+}(M^3) \rightarrow C(M^3,\GLp_3(\R))\]
that we will call the {\em derivative crossed homomorphism}.  In a
suitable sense, it encodes the action of $\Diff^{+}(M^3)$ on $\Triv(M^3)$.
The derivative crossed homomorphism depends on a choice of a base trivialization
$\sigma_0 \in \Triv(M^3)$ that we fix once and for all.  Now consider $f \in \Diff^{+}(M^3)$.
We have $\sigma_0^f \in \Triv(M^3)$, and as we noted above the topological group $C(M^3,\GLp_3(\R))$ acts
simply transitively on $\Triv(M^3)$.  It follows that there exists a unique $\phi_f \in C(M^3,\GLp_3(\R))$
such that
\[\sigma_0^f = \sigma_0 \cdot \phi_f.\]
We define $\cD(f) = \phi_f^{-1}$.  Here the inverse refers to the group structure
on the space $C(M^3,\GLp_3(\R))$ induced by the group structure on $\GLp_3(\R)$.  
The inverse will be needed to make $\cD$ a crossed homomorphism -- if you
examine the formulas below, you will see that without it $\cD$ would be a crossed anti-homomorphism.

To check that $\cD$ is indeed a crossed homomorphism, note that for
$f_1,f_2 \in \Diff^{+}(M^3)$ we have
\[\sigma_0^{f_1 f_2} = \sigma_0 \cdot \cD(f_1 f_2)^{-1}\]
and
\begin{align*}
\sigma_0^{f_1 f_2} &= \left(\sigma_0^{f_1}\right)^{f_2} \\
&= \left(\sigma_0 \cdot \cD\left(f_1\right)^{-1}\right)^{f_2} \\
&= \sigma_0^{f_2} \cdot \left(\cD\left(f_1\right)^{f_2}\right)^{-1} \\
&= \sigma_0 \cdot \cD(f_2)^{-1} \cdot \left(\cD\left(f_1\right)^{f_2}\right)^{-1}.
\end{align*}
Here again the inverses refer to the group structure on the space $C(M^3,\GLp_3(\R))$ induced by the group structure on $\GLp_3(\R)$.  We thus have
\[\cD(f_1 f_2) = \cD\left(f_1\right)^{f_2} \cdot \cD(f_2),\]
as desired.

\paragraph{Homotopy classes.}
We now pass to homotopy.  Let $[\sigma]$ denote the homotopy class of $\sigma \in \Triv(M^3)$ and let
\[\HTriv(M^3) = \Set{$[\sigma]$}{$\sigma \in \Triv(M^3)$},\]
and let $[\phi]$ denote the homotopy class of $\phi \in C(M^3,\GLp_3(\R))$ and let
\[[M^3,\GLp_3(\R)] = \Set{$[\phi]$}{$\phi \in C(M^3,\GLp_3(\R))$}.\]
Finally, let $\Mod(M^3) = \pi_0(\Diff^{+}(M^3))$ denote the mapping class group of $M^3$, and
for $f \in \Diff^{+}(M^3)$ let $[f] \in \Mod(M^3)$ denote its isotopy class.
The group structures of $\Diff^{+}(M^3)$ and $C(M^3,\GLp_3(\R))$ induce group
structures on $\Mod(M^3)$ and $[M^3,\GLp_3(\R)]$, and the right actions
of $\Diff^{+}(M^3)$ and $C(M^3,\GLp_3(\R))$ on $\Triv(M^3)$ induce right actions
of $\Mod(M^3)$ and $[M^3,\GLp_3(\R)]$ on $\HTriv(M^3)$ that we will continue
to write with superscripts and $\cdot$'s, respectively.  For $[f] \in \Mod(M^3)$ and $[\phi] \in [M^3,\GLp_3(\R)]$
and $[\sigma] \in \HTriv(M^3)$, we still have the relationship
\[\left([\sigma] \cdot [\phi]\right)^{[f]} = [\sigma]^{[f]} \cdot [\phi]^{[f]}.\]
Finally, the derivative crossed homomorphism $\cD\colon \Diff^{+}(M^3) \rightarrow C(M^3,\GLp_3(\R))$
descends to a derivative crossed homomorphism
\[\fD\colon \Mod(M^3) \rightarrow [M^3,\GLp_3(\R)]\]
whose characteristic property is that
\[[\sigma_0]^{[f]} = [\sigma_0] \cdot \fD([f])^{-1} \quad \text{for $[f] \in \Mod(M^3)$}.\]

\section{The twisting crossed homomorphism}
\label{section:twisting}

Just like in the last section, let $M^3$ be a closed oriented $3$-manifold.  Fix some $\sigma_0 \in \Triv(M^3)$, and let $\fD\colon \Mod(M^3) \rightarrow [M^3,\GLp_3(\R)]$
be the associated derivative crossed homomorphism.  

\paragraph{Twisting crossed homomorphism.}
The group $\GLp_3(\R)$ deformation retracts to its maximal 
compact subgroup $\SO(3) \cong \RP^3$.
We thus have $\pi_1(\GLp_3(\R)) \cong \Z/2$.  The $\pi_1$-functor therefore induces a group homomorphism
\[h\colon [M^3,\GLp_3(\R)] \longrightarrow \Hom(\pi_1(M^3),\Z/2) = \HH^1(M^3;\Z/2).\]
The group $\Mod(M_3)$ acts on the right on both $[M^3,\GLp_3(\R)]$ and $\HH^1(M^3;\Z/2)$, and $h$ commutes with
this action in the sense that for $[f] \in \Mod(M^3)$ and $[\phi] \in [M^3,\GLp_3(\R)]$ we have
\[h\left(\left[\phi\right]^{\left[f\right]}\right) = h\left(\left[\phi\right]\right)^{\left[f\right]}.\]
This implies that the composition of $h$ with the derivative crossed homomorphism $\fD$ is a crossed homomorphism
\[\fT\colon \Mod(M^3) \longrightarrow \HH^1(M^3;\Z/2)\]
that we will call the {\em twisting crossed homomorphism}.  Since the twist subgroup $\Tw(M^3) < \Mod(M^3)$ acts trivially
on $\pi_1(M^3)$, the restriction of $\fT$ to $\Tw(M^3)$ is a homomorphism (not just a crossed homomorphism).  

\paragraph{Effect on sphere twists.}
The following lemma shows how to calculate $\fT$ on a sphere twist:

\begin{lemma}
\label{lemma:spheretwist}
Let $S$ be an embedded $2$-sphere in $M^3$.  Then $\fT(T_S) \in \HH^1(M^3;\Z/2)$ is the cohomology class
that is Poincar\'{e} dual to $[S] \in \HH_2(M^3;\Z/2)$.
\end{lemma}
\begin{proof}
Identify $S$ with $S^2 \subset \R^3$.
Recall that $T_S$ is constructed from a loop $\ell\colon [0,1] \rightarrow \SO(3)$ with $\ell(0) = \ell(1) = \text{id}$
that generates $\pi_1(\SO(3),\text{id}) \cong \Z/2$.  This generator rotates $S$ by a full
twist about an axis, and $T_S$ is represented by a diffeomorphism $\tau$ that is the identity outside a tubular neighborhood
$U \cong S \times [0,1]$ of $S$, and on $U$ is defined by $\tau(s,t) = (\ell(t) \cdot s, t)$.  Let $p_0 \in S$ be one
of the two intersection points of the axis of rotation defining $\ell$ with $S$.

Consider a smoothly embedded closed curve $\gamma\colon S^1 \rightarrow M^3$.  Homotoping $\gamma$, we can assume that it only intersects $U$ in segments of the form $p_0 \times [0,1]$ (which it might traverse in either direction).  It follows that
$\tau$ fixes $\gamma$ pointwise.  Let $\cD\colon \Diff^{+}(M^3) \rightarrow C(M^3,\GLp_3(\R))$ be the derivative
crossed homomorphism that descends to $\fD\colon \Mod(M^3) \rightarrow [M^3,\GLp_3(\R)]$ when we pass
to homotopy.  The composition
\[[0,1] \xrightarrow{\gamma} M^3 \xrightarrow{\cD(\tau)} \GLp_3(\R)\]
is a loop whose image in $\pi_1(\GLp_3(\R)) \cong \Z/2$ represents $\fT(T_S)([\gamma])$.  Examining the definitions, we
see that this element of $\pi_1(\GLp_3(\R)) \cong \Z/2$ simply counts the number of times $\gamma$ traverses
$p_0 \times [0,1]$, which equals the $\Z/2$-algebraic intersection number of $\gamma$ with $S$.  The lemma follows.
\end{proof}

\paragraph{Connect sums of $S^2 \times S^1$.}
We now specialize this to the connect sum $M_n$ of $n$ copies of $S^2 \times S^1$.  Recall from
Theorem \ref{theorem:laudenbach} that $\Tw(M_n)$ is generated by the sphere twists about
the core spheres $S^2 \times \ast$ of the $n$ summands $S^2 \times S^1$ of $M_n$.  These clearly commute
with each other and are Poincar\'{e} dual to a basis for $\HH^1(M^3;\Z/2)$, so
Lemma \ref{lemma:spheretwist} implies the following:

\begin{corollary}
\label{corollary:identifytwist}
The twisting crossed homomorphism $\fT\colon \Mod(M_n) \rightarrow \HH^1(M_n;\Z/2)$ restricts to an isomorphism
$\Tw(M_n) \cong \HH^1(M_n;\Z/2)$.
\end{corollary}

In particular, we recover Laudenbach's theorem \cite{LaudenbachSpheres} saying that $\Tw(M_n) \cong (\Z/2)^n$.  
We actually get more: since $\fT\colon \Mod(M_n) \rightarrow \HH^1(M_n;\Z/2)$ is a crossed homomorphism, 
the isomorphism $\Tw(M_n) \cong \HH^1(M_n;\Z/2)$ in
Corollary \ref{corollary:identifytwist} is an isomorphism of $\Mod(M_n)$-modules, where 
$\Mod(M_n)$ acts on its normal subgroup $\Tw(M_n)$ by conjugation.  

\paragraph{Summary.}
Combining Corollary \ref{corollary:identifytwist} with Lemma \ref{lemma:crossedsemidirect} and the
exact sequence
\[1 \longrightarrow \Tw(M_n) \longrightarrow \Mod(M_n) \longrightarrow \Out(F_n) \longrightarrow 1\]
from Theorem \ref{theorem:laudenbach}, we conclude the following:

\begin{theorem}
\label{theorem:weakmain}
Let $[\sigma_0]$ be the homotopy class of a trivialization $\sigma_0$ of the tangent bundle of $M_n$
and let $\fT\colon \Mod(M_n) \rightarrow \HH^1(M_n;\Z/2)$ be the associated twisting crossed homomorphism.  The
following then hold:
\begin{compactitem}
\item $\Mod(M_n) = \Tw(M_n) \rtimes \ker(\fT)$.
\item $\Tw(M_n) \cong \HH^1(M_n;\Z/2)$ as a $\Mod(M_n)$-module.
\item $\ker(\fT) \cong \Out(F_n)$.
\end{compactitem}
\end{theorem}

This is almost Theorem \ref{maintheorem:split}.  All that is missing is the fact that
$\ker(\fT) \cong \Out(F_n)$ is the $\Mod(M_n)$-stabilizer of $[\sigma_0]$, which we will
prove in the next section (see Corollary \ref{corollary:stabilizer}).

\section{\texorpdfstring{$\Out(F_n)$}{Out(Fn)} acts trivially on homotopy classes of trivializations}
\label{section:completing}

In this section, we prove the following.

\begin{lemma}
\label{lemma:fixtrivial}
Let $[\sigma_0]$ be the homotopy class of a trivialization $\sigma_0$ of the tangent bundle of $M_n$
and let $\fT\colon \Mod(M_n) \rightarrow \HH^1(M_n;\Z/2)$ be the associated twisting crossed homomorphism.  Then
$\ker(\fT)$ fixes $[\sigma_0]$.
\end{lemma}

Since the $\Mod(M_n)$-stabilizer of $[\sigma_0]$ is clearly contained in $\ker(\fT)$, this implies the following:

\begin{corollary}
\label{corollary:stabilizer}
Let $[\sigma_0]$ be the homotopy class of a trivialization $\sigma_0$ of the tangent bundle of $M_n$
and let $\fT\colon \Mod(M_n) \rightarrow \HH^1(M_n;\Z/2)$ be the associated twisting crossed homomorphism.
Then $\ker(\fT)$ is the $\Mod(M_n)$-stabilizer of $[\sigma_0]$.
\end{corollary}

As we noted at the end of \S \ref{section:twisting}, Corollary \ref{corollary:stabilizer} together with 
Theorem \ref{theorem:weakmain} implies Theorem \ref{maintheorem:split} from the introduction.

\begin{proof}[Proof of Lemma \ref{lemma:fixtrivial}]
Let $G = \ker(\fT)$, so by Theorem \ref{theorem:weakmain} we have $G \cong \Out(F_n)$.  Let
$\fD\colon \Mod(M_n) \rightarrow [M_n,\GLp_3(\R)]$ be the derivative crossed homomorphism
associated to $[\sigma_0]$.  Recall that $\GLp_3(\R)$ is homotopy equivalent to
$\RP^3$.  Let $\tGLp_3(\R)$ be the universal cover of $\GLp_3(\R)$, so $\tGLp_3(\R)$ is homotopy
equivalent to $S^3$.  For $[f] \in G$, we know that $\fD([f]) \in [M_n,\GLp_3(\R)]$ induces
the trivial map on $\pi_1$, so we can lift $\fD([f])$ to an element of $[M_n, \tGLp_3(\R)]$.
Though there are two distinct lifts to $\tGLp_3(\R)$ of a map $M_n \rightarrow \GLp_3(\R)$ that induces
the trivial map on $\pi_1$ (the two lifts correspond to a choice of a lift of a basepoint), these two lifts
are homotopic via a homotopy corresponding to right-multiplication by a path in $\tGLp_3(\R)$ from
the identity to the other element of $\tGLp_3(\R)$ projecting to the identity in $\GLp_3(\R)$.  From
this, we see that in fact $\fD$ lifts to a crossed homomorphism
$\tfD\colon G \rightarrow [M_n,\tGLp_3(\R)]$.

Since $\tGLp_3(\R)$ is homotopy equivalent to $S^3$, elements of $[M_n,\tGLp_3(\R)]$ are classified
by their degree, i.e., as a group we have $[M_n,\tGLp_3(\R)] \cong \Z$.  What is more, since orientation-preserving
diffeomorphisms of $M_n$ act by degree $1$, the action of $\Mod(M_n)$ on $[M_n,\tGLp_3(\R)] \cong \Z$ is trivial.
It follows that the crossed homomorphism 
\[\tfD\colon G \rightarrow [M_n,\tGLp_3(\R)] \cong \Z\] 
is an actual homomorphism (not just a crossed homomorphism).  Since the abelianization of $G \cong \Out(F_n)$
is torsion,\footnote{Here is a quick proof of this classical fact.  Since $\Aut(F_n)$ surjects onto $\Out(F_n)$, it is enough
to prove that $(\Aut(F_n))^{\ab}$ is torsion.  Recall the generators $L_{ij}$ and $R_{ij}$ and $I_i$ for $\Aut(F_n)$
from Claim \ref{claims:surjective} of the proof of Theorem \ref{theorem:laudenbach}.  It is enough to prove that each of these generators maps
to an element of $(\Aut(F_n))^{\ab}$ of finite order.  This is trivial for the order-$2$ elements $I_i$, so we
must just deal with the $L_{ij}$ and $R_{ij}$.  The key observation is that 
\[I_j L_{ij} I_j^{-1} = L_{ij}^{-1} \quad \text{and} \quad I_j R_{ij} I_j^{-1} = R_{ij}^{-1}.\]
The elements $L_{ij}$ and $I_j^{-1} L_{ij} I_j = L_{ij}^{-1}$ map to the same element of $(\Aut(F_n))^{\ab}$, so $L_{ij}^{2}$ must
map to $0$.  Similarly, $R_{ij}^2$ maps to $0$ in $(\Aut(F_n))^{\ab}$.}
we deduce that in fact $\tfD$ is trivial.  This implies that $\fD|_G$ is also trivial.  Since
$\fD$ encodes the action of $\Mod(M_n)$ on $[\sigma_0] \in \HTriv(M_n)$ (see \S \ref{section:derivative}), this implies that
$G$ acts trivially on $[\sigma_0]$, as desired.
\end{proof}

\begin{footnotesize}
\noindent
\begin{tabular*}{\linewidth}[t]{@{}p{\widthof{School of Mathematics \& Statistics}+0.3in}@{}p{\widthof{Department of Mathematics}+0.3in}@{}p{\linewidth - \widthof{Department of Mathematics} - \widthof{School of Mathematics \& Statistics}- 0.6in}@{}}
{\raggedright
Tara Brendle\par
School of Mathematics \& Statistics\par
University of Glasgow\par
University Place\par
Glasgow G12 8QQ\par
UK\par
{\tt tara.brendle@glasgow.ac.uk}}
&
{\raggedright
Nathan Broaddus\par
Department of Mathematics\par
Ohio State University\par
231 W. 18th Ave.\par
Columbus, OH 43210\par
USA\par
{\tt broaddus.9@osu.edu}}
&
{\raggedright
Andrew Putman\par
Department of Mathematics\par
University of Notre Dame \par
255 Hurley Hall\par
Notre Dame, IN 46556\par
USA\par
{\tt andyp@nd.edu}}
\end{tabular*}
\end{footnotesize}

\end{document}